\documentclass{article}  

\usepackage{amsmath,amsfonts,amssymb,amsthm}
\usepackage{float}
\usepackage{graphicx}
\usepackage{a4wide}

\theoremstyle{plain}
\newtheorem{theorem}{Theorem}
\newtheorem{lemma}{Lemma}
\newtheorem{corollary}{Corollary}
\theoremstyle{definition}
\newtheorem{definition}{Definition}
\newtheorem{example}{Example}
\newtheorem{remark}{Remark}

\newcommand\blfootnote[1]{%
\begingroup
\renewcommand\thefootnote{}\footnote{#1}%
\addtocounter{footnote}{-1}%
\endgroup
}

\begin{document}	

\title{A class of fractional differential equations via power\\ 
non-local and non-singular kernels: existence,\\ 
uniqueness and numerical approximations\blfootnote{This is a preprint 
of a paper whose final form is published in \emph{Physica D: Nonlinear Phenomena} (ISSN 0167-2789).
Submitted 19-Jan-2023; revised 15-May-2023; accepted for publication 11-Oct-2023.}}

\author{Hanaa Zitane\\
\texttt{h.zitane@ua.pt}
\and Delfim F. M. Torres\thanks{Corresponding author.}\\
\texttt{delfim@ua.pt}}

\date{Center for Research and Development in Mathematics and Applications (CIDMA),
Department of Mathematics, University of Aveiro, 3810-193 Aveiro, Portugal}

\maketitle

% -----------------------------------------------

\begin{abstract}
We prove a useful formula and new properties 
for the recently introduced power fractional calculus 
with non-local and non-singular kernels. In particular, 
we prove a new version of Gronwall's inequality involving 
the power fractional integral; and we establish existence 
and uniqueness results for nonlinear power fractional 
differential equations using fixed point techniques. Moreover, 
based on Lagrange polynomial interpolation, we develop 
a new explicit numerical method in order to approximate 
the solutions of a rich class of fractional differential equations. 
The approximation error of the proposed numerical scheme is analyzed. 
For illustrative purposes, we apply our method to a fractional differential 
equation for which the exact solution is computed, as well as to a 
nonlinear problem for which no exact solution is known. The numerical
simulations show that the proposed method is very efficient, 
highly accurate and converges quickly.

\medskip


\medskip

\end{abstract}

% -----------------------------------------------------------------------

\section{Introduction}
\label{Sec:1}

Over the last decades, fractional differential equations (FDEs) 
have been used to model a large variety of physical, biological, 
and engineering problems \cite{Baleanu1,Srivastava}. Often, 
since most dynamical systems involve memory or hereditary effects, 
the non-locality properties of the fractional derivatives 
make them more accurate in modeling when compared with 
the classical local operators. That gave rise to the 
introduction of different kinds of non-local fractional derivatives 
with non-singular kernels \cite{ALRefai1,AtanBal,CapFab,Hattaf},
e.g., Caputo--Fabrizio \cite{CapFab}, Atangana--Baleanu \cite{AtanBal},  
weighted Atangana--Baleanu \cite{ALRefai1}, and Hattaf fractional 
derivatives \cite{Hattaf}. 

In 2022, a generalized version of all the previous non-local fractional
derivatives with non-singular kernels was introduced: the so-called 
power fractional derivative (PFD) \cite{PowerDerivative}. 
PFDs are based on the  generalized power Mittag--Leffler function, 
which contains a key ``power'' parameter $p$ that plays a very important 
role by enabling researchers, engineers and scientists, to select 
the adequate fractional derivative that models more accurately 
the real world phenomena under study. The authors of \cite{PowerDerivative}
presented the basic properties of the new power fractional derivative and integral. 
Moreover, they provided the Laplace transform corresponding to the PFD, 
which is then applied to solve a class of linear fractional differential equations.

The question of existence and uniqueness of nonlinear FDEs, 
as well as their various applications, have been 
discussed by many researchers: see, for instance, 
\cite{ALRefai,Cheneke,KHattaf,Jarad,Sene} 
and references cited therein. Analyzing the literature,
one may conclude that Gronwall's inequality 
and its extensions are one of the most fundamental tools in all such results. 
Indeed, several versions of this classical inequality, involving fractional integrals 
with non-singular kernels, have been provided in order to develop the quantitative 
and qualitative properties of the fractional differential equations to be investigated 
\cite{Gronwall,KHattaf,Jarad}. For example, in \cite{KHattaf}, Hattaf et al. establish 
a Gronwall's inequality in the framework of generalized Hattaf fractional integrals, 
while in \cite{Jarad} Alzabut et al. prove a Gronwall's inequality via Atangana--Baleanu 
fractional integrals. 

Motivated by the proceeding, the first main purpose of the present work 
is to derive a new version of Gronwall's inequality, as well as to study 
the existence and uniqueness of solutions for nonlinear fractional differential 
equations in the framework of more general power fractional operators 
with non-local and non-singular kernels. On the other hand, we develop
an appropriate numerical method to deal with power differential equations.

Numerical methods have been recognized as indispensable 
in fractional calculus \cite{MR3727142}. They provide 
powerful mathematical tools to solve nonlinear ordinary 
differential equations and fractional differential equations 
modeling complex real phenomena. Numerical methods are generally 
applied to predict the behavior of dynamical systems when all 
the used analytical methods fail, as it often the case. 
Various numerical schemes have been developed to approximate 
the solutions of different types of fractional differential equations 
with singular and non-singular kernels 
\cite{Hattaf2,HatNum,Shiri,ToufAtan}. 
For example, in \cite{Hattaf2} a numerical scheme, that recovers the 
classical Euler's method for ordinary differential equations, is proposed, 
in order to obtain numerical solutions of FDEs with generalized 
Hattaf fractional derivatives; in \cite{Shiri} collocation and 
predictor-corrector methods on piece-wise polynomial spaces 
are developed to solve tempered FDEs with Caputo fractional 
derivatives; while in \cite{ToufAtan} a numerical approximation 
for FDEs with Atangana--Baleanu fractional derivatives is investigated. 
However, to the best of our knowledge, no numerical methods have yet 
been developed to solve FDEs in the framework of power fractional derivatives. 
Consequently, the second main purpose of our work is to develop 
a new numerical scheme for approximating the solutions of such  
general and powerful differential equations.

The remainder of this article is organized as follows. 
Section~\ref{Sec2} states the necessary preliminaries, 
including the definitions of power fractional derivative 
and integral in the Caputo sense. In Section~\ref{Sec3}, 
we establish new and important formula and properties 
for the power fractional operators with non-local 
and non-singular kernels that we will need in the sequel. 
Section~\ref{Sec4} deals with a new more general version 
of Gronwall's inequality for the power fractional integral. 
Then we proceed with Section~\ref{Sec5}, which is devoted 
to the existence and uniqueness of solutions to FDEs involving PFDs. 
Section~\ref{NumAnaly} introduces a new numerical scheme with its 
error analysis, allowing one to investigate, in practical terms,
power FDEs. Applications and numerical simulations 
of our main results are given in Section~\ref{Sec7}. We end with
Section~\ref{sec:conc} of conclusions.

% -----------------------------------------------------------------------

\section{Essential preliminaries and notations}
\label{Sec2}

In this section, we recall necessary definitions and results 
from the literature that will be useful in the sequel. 
Throughout this paper, $g\in H^{1}(a,b)$ is a sufficiently 
smooth function on $[a, b]$, with $a, b \in \mathbb{R}$, and  
$H^{1}(a,b)$ is the Sobolev space of order one. Also, 
$AC([ a, b])$ denotes the space of absolutely continuous functions 
$u$ on $[ a, b]$ endowed with the norm 
$\Arrowvert  u \Arrowvert= \underset{t\in[ a, b ]}{\sup}|u(t)|$. 
In addition, we adopt the notations
$$
\phi(\alpha):=\dfrac{1-\alpha}{N(\alpha)}, 
\quad \psi(\alpha):=\dfrac{\alpha}{N(\alpha)},
$$ 
where $\alpha \in [0, 1)$ and $N(\alpha)$ is a normalization positive 
function obeying $N(0)=N(1^{-})=1$ with 
$N(1^{-})=\underset{\alpha \rightarrow 1^{-}}{\lim}N(\alpha)$.

\begin{definition}[See \cite{PowerDerivative}]
\label{PowerMittag} 
The power Mittag--Leffler function is given by
\begin{equation}
\label{PMF}
{}^p \! E_{k,l}(s)=\displaystyle\sum_{n=0}^{+\infty}
\dfrac{(s\ln p)^{n}}{\Gamma(k n+l)},
\quad s\in\mathbb{C},
\end{equation}
where $k >0$, $l>0$, $p>0$,
and $\Gamma(\cdot)$ is the Gamma function \cite{mittag}.
\end{definition}

\begin{remark}
The term $\ln(p)$ that is introduced in Definition~\ref{PowerMittag} 
of power Mittag-Leffler function  
${}^p \! E_{k,l}(\cdot)$ allows, by taking particular cases, to obtain several important 
functions available in the literature, for example, the Mittag--Leffler function 
of one parameter ${}^e \! E_{k,1}(\cdot)$ \cite{mittag}, 
the Wiman function ${}^e \! E_{k,l}(\cdot)$ \cite{Wiman}, 
and those introduced by Prabhakar \cite{Prabhakar:71} and \cite{Salim:2009}.
\end{remark}

\begin{definition}[See \cite{PowerDerivative}]
\label{PowerDerivative}
Let $\alpha \in [0, 1)$, $\beta > 0$, $p>0$, and $g\in H^{1}(a,b)$. 
The power fractional derivative (PFD) of order $\alpha$, in the Caputo sense, 
of a function $g$ with respect to the weight function $\omega$, is defined by
\begin{equation}
\label{PFD}
{}^p{}^C \!D_{a,t,\omega}^{\alpha,\beta,p}g(t)=\dfrac{1}{\phi(\alpha)}\dfrac{1}{\omega(t)}
\int_{a}^{t} {}^p \! E_{\beta,1}\left(-\mu_{\alpha}(t-s)^{\beta}\right)(\omega g)'(s)
\, \mathrm{d}s,
\end{equation}
where $\mu_{\alpha}:=\dfrac{\alpha}{1-\alpha}$ 
and $\omega \in C^{1}([a,b])$ with $\omega>0$ on $[a,b]$.
\end{definition}

\begin{remark}
PFD is a fractional derivative with non-singular kernel 
while the classical Caputo fractional derivative
is a fractional operator with singular kernel. Therefore, PFDs belong 
to a different family and do not include Caputo derivatives as special cases.
\end{remark}

\begin{remark}
Note that the PFD \eqref{PFD} includes many  interesting fractional 
derivatives that exist in the literature, such as:
\begin{itemize}
\item if $p=e$, then we retrieve the generalized Hattaf
fractional derivative \cite{Hattaf} given by 
$$
{}^p{}^C \!D_{a,t,\omega}^{\alpha,\beta,e}g(t)=\dfrac{1}{\phi(\alpha)}
\dfrac{1}{\omega(t)} \int_{a}^{t} E_{\beta,1}\left(-\mu_{\alpha}
(t-s)^{\beta}\right)(\omega g)'(s)\, \mathrm{d}s;
$$
\item if $\beta=\alpha$, $p=e$ and $\omega(t) \equiv 1$, 
then we obtain the Atangana--Baleanu fractional derivative 
\cite{AtanBal} defined as
$$
{}^p{}^C \!D_{a,t,1}^{\alpha,\alpha,e}g(t)
=\dfrac{1}{\phi(\alpha)}\int_{a}^{t} 
E_{\alpha,1}\left(-\mu_{\alpha}
(t-s)^{\alpha}\right) g'(s)\, \mathrm{d}s;
$$
\item if $\beta=1$, $p=e$ and $\omega(t) \equiv 1$, then 
we get the Caputo--Fabrizio fractional derivative \cite{CapFab} given by 
$$
{}^p{}^C \!D_{a,t,1}^{\alpha,1,e}g(t)=\dfrac{1}{\phi(\alpha)}
\int_{a}^{t} \exp\left(-\mu_{\alpha}(t-s)\right) g'(s)\, \mathrm{d}s.
$$
\end{itemize}
\end{remark}

The power fractional integral associated with the power fractional derivative 
${}^p{}^C \!D_{a,t,\omega}^{\alpha,\beta,p}$ is given 
in Definition~\ref{PowerDIntegral}.

\begin{definition}[See \cite{PowerDerivative}]	
\label{PowerDIntegral}	
The power fractional integral (PFI) of order $\alpha$, 
of a function $g$ with respect to the weight function $\omega$, is given by
\begin{equation}
\label{PFI}
{}^p\!I_{a,t,\omega}^{\alpha,\beta,p}g(t)
=\phi(\alpha)g(t)+\ln p\cdot\psi(\alpha) {}^R\!{}^L\!I_{a,\omega}^{\beta}g(t),
\end{equation}
where ${}^R\!{}^L\!I_{a,\omega}^{\beta}$ denotes the standard weighted 
Riemann--Liouville fractional integral of order $\beta$ given by
$$
{}^R\!{}^L\!I_{a,\omega}^{\beta}g(t)
=\dfrac{1}{\Gamma(\beta)}\dfrac{1}{\omega(t)}
\int_{a}^{t}(t-s)^{\beta-1}(\omega g)(s)\, \mathrm{d}s.
$$	
\end{definition}

\begin{remark}
For $p=e$, the PFI \eqref{PFI} coincides with the generalized 
fractional integral introduced in~\cite{Hattaf}.
\end{remark}

The Gronwall's inequality in the framework of the weighted Riemann--Liouville 
fractional integral is given in \cite{KHattaf}.

\begin{lemma}[See \cite{KHattaf}]
\label{GHattaf}
Suppose $\beta>0$, $h$ and $u$ are non-negative 
and locally integrable functions on $[a,b)$, 
and $v$ is a non-negative, non-decreasing,
and continuous function on $[a,b)$ satisfying 
$v(t) \leq \lambda$, where $\lambda$ is a constant.  
If
$$ 
h(t)\leq u(t)+v(t){}^R\!{}^L\!I_{a,\omega}^{\beta}h(t),
$$
then
$$
h(t)\leq u(t)+\int_{a}^{t}\sum_{n=1}^{+\infty} 
\dfrac{(v(t))^{n}}{\Gamma(n\beta)}
(t-s)^{n\beta-1}u(s)\, \mathrm{d}s.
$$ 
\end{lemma}

% -----------------------------------------------------------------------

\section{New properties of the power fractional operators}
\label{Sec3}

In this section, we establish a new important formula and properties 
for the power fractional operators. They will be useful in the sequel
to achieve the main goals formulated in Section~\ref{Sec:1}.

\begin{lemma}
\label{impo1}
The power Mittag--Leffler function ${}^p \! E_{k,l}(s)$ 
is locally uniformly convergent for any $s\in\mathbb{C}$.
\end{lemma}

\begin{proof}
The proof is similar to the proof of Theorem~1 of \cite{PowerDerivative}.
\end{proof}

We prove a new formula for the power fractional derivative 
in the form of an infinite series of the standard weighted 
Riemann--Liouville fractional integral, which brings out 
more clearly the non-locality properties of the fractional 
derivative and, for certain computational purposes, 
is easier to handle than the original formula \eqref{PFD}.

\begin{lemma}
\label{key1}
The power fractional derivative 
${}^p{}^C \!D_{a,t,\omega}^{\alpha,\beta,p}$ 
can be expressed as follows:
\begin{equation*}
{}^p{}^C \!D_{a,t,\omega}^{\alpha,\beta,p}g(t)
=\dfrac{1}{\phi(\alpha)}\displaystyle\sum_{n=0}^{+\infty}\left(-\mu_{\alpha}
\ln p\right)^{n} {}^R\!{}^L\!I_{a,\omega}^{\beta n+1}
\left(\dfrac{(\omega g)'}{\omega}\right)(t),
\end{equation*}
where the series converges locally and uniformly in $t$ for any 
$a$, $\alpha$, $\beta$, $p$, $\omega$ and $g$ verifying 
the conditions laid out in Definition~\ref{PowerDerivative}.
\end{lemma}

\begin{proof}
The power Mittag--Leffler function ${}^p \! E_{k,l}(s)$ 
is an entire function of $s$. Since it is locally  uniformly 
convergent in the whole complex plane (see Lemma~\ref{impo1}), 
then the PFD may be rewritten as follows:
\begin{equation*}	
\begin{split}
{}^p{}^C \!D_{a,t,\omega}^{\alpha,\beta,p}g(t)
&=\dfrac{1}{\phi(\alpha)}\dfrac{1}{\omega(t)}
\displaystyle\sum_{n=0}^{+\infty}\dfrac{\left(-\mu_{\alpha}
\ln p\right)^{n}}{\Gamma(\beta n+1)}\int_{a}^{t}
(t-x)^{\beta n}(\omega g)'(x)\, \mathrm{d}x \\
&=\dfrac{1}{\phi(\alpha)}\displaystyle\sum_{n=0}^{+\infty}
\left(-\mu_{\alpha}\ln p\right)^{n}\dfrac{1}{\Gamma(\beta n+1)}
\dfrac{1}{\omega(t)}\int_{a}^{t}
(t-x)^{\beta n}(\omega g)'(x)\, \mathrm{d}x\\
&=\dfrac{1}{\phi(\alpha)}\displaystyle\sum_{n=0}^{+\infty}
\left(-\mu_{\alpha}\ln p\right)^{n} {}^R\!{}^L\!I_{a,\omega}^{\beta n+1}
\left(\dfrac{(\omega g)'}{\omega}\right)(t),
\end{split}	
\end{equation*}	
which completes the proof.
\end{proof}

\begin{theorem}
\label{PFDandPFI} 
Let $\alpha \in [0, 1)$, $\beta > 0$, $p>0$,
and $g \in H^{1}(a,b)$. Then, 
\begin{itemize}
\item[(i)] ${}^p{}^C \!D_{a,t,\omega}^{\alpha,\beta,p}
\left({}^p\!I_{a,t,\omega}^{\alpha,\beta,p}g\right)(t)
=g(t)-\dfrac{(\omega g)(a)}{\omega(t)}$;
\item[(ii)] ${}^p\!I_{a,t,\omega}^{\alpha,\beta,p}
\left({}^p{}^C \!D_{a,t,\omega}^{\alpha,\beta,p}g\right)(t)
=g(t)-\dfrac{(\omega g)(a)}{\omega(t)}$.
\end{itemize}	
\end{theorem}

\begin{proof}
We begin by proving $(i)$. According to Lemma~\ref{key1}, one has
$$
{}^p{}^C \!D_{a,t,\omega}^{\alpha,\beta,p}
\left({}^p\!I_{a,t,\omega}^{\alpha,\beta,p}g\right)(t)
=\dfrac{1}{\phi(\alpha)}\displaystyle\sum_{n=0}^{+\infty}
\left(-\mu_{\alpha}\ln p\right)^{n} {}^R\!{}^L\!I_{a,\omega}^{\beta n+1}
\left(\dfrac{\left(\omega \left({}^p\!I_{a,t,\omega}^{\alpha,\beta,p}
g\right)\right)'}{\omega}\right)(t).
$$
From Definition~\ref{PowerDIntegral}, it follows that
\begin{equation*}	
\begin{split}
{}^p{}^C \!D_{a,t,\omega}^{\alpha,\beta,p}
\left({}^p\!I_{a,t,\omega}^{\alpha,\beta,p}g\right)(t)
&=\dfrac{1}{\phi(\alpha)}\displaystyle\sum_{n=0}^{+\infty}
\left(-\mu_{\alpha}\ln p\right)^{n} {}^R\!{}^L\!I_{a,\omega}^{\beta n+1}
\left[\dfrac{\phi(\alpha)(\omega g)'}{\omega}+\dfrac{\ln p
\cdot\psi(\alpha)(\omega {}^R\!{}^L\!I_{a,\omega}^{\beta}g)'}{\omega}\right](t)\\
&=\displaystyle\sum_{n=0}^{+\infty}\left(-\mu_{\alpha}
\ln p\right)^{n} {}^R\!{}^L\!I_{a,\omega}^{\beta n+1}
\left(\dfrac{(\omega g)'}{\omega}\right)(t)\\
&\qquad+ \mu_{\alpha}\ln p\displaystyle\sum_{n=0}^{+\infty}
\left(-\mu_{\alpha}\ln p\right)^{n}{}^R\!{}^L\!I_{a,\omega}^{\beta n+1}
\left(\dfrac{(\omega {}^R\!{}^L\!I_{a,\omega}^{\beta}g)'}{\omega}\right)(t).
\end{split}	
\end{equation*}
Therefore,
\begin{equation*}
\begin{split}
{}^p{}^C \!D_{a,t,\omega}^{\alpha,\beta,p}
\left({}^p\!I_{a,t,\omega}^{\alpha,\beta,p}g\right)(t)
&=\displaystyle\sum_{n=0}^{+\infty}\left(-\mu_{\alpha}
\ln p\right)^{n}\left[{}^R\!{}^L\!I_{a,\omega}^{\beta n}
g(t)-(\omega g)(a){}^R\!{}^L\!I_{a,\omega}^{\beta n}
\left(\dfrac{1}{\omega}\right)(t)\right]\\
&\quad-\displaystyle\sum_{n=0}^{+\infty}
\left(-\mu_{\alpha}\ln p\right)^{n+1}
\left[{}^R\!{}^L\!I_{a,\omega}^{\beta (n+1)}
g(t)-(\omega g)(a){}^R\!{}^L\!I_{a,\omega}^{\beta (n+1)}
\left(\dfrac{1}{\omega}\right)(t)\right]\\
&=\displaystyle\sum_{n=0}^{+\infty}\left(-\mu_{\alpha}
\ln p\right)^{n}\left[{}^R\!{}^L\!I_{a,\omega}^{\beta n}
g(t)-(\omega g)(a){}^R\!{}^L\!I_{a,\omega}^{\beta n}
\left(\dfrac{1}{\omega}\right)(t)\right]\\
&\quad-\displaystyle\sum_{n=1}^{+\infty}
\left(-\mu_{\alpha}\ln p\right)^{n}
\left[{}^R\!{}^L\!I_{a,\omega}^{\beta n}
g(t)-(\omega g)(a){}^R\!{}^L\!I_{a,\omega}^{\beta n}
\left(\dfrac{1}{\omega}\right)(t)\right]\\
&={}^R\!{}^L\!I_{a,\omega}^{0}g(t)
-(\omega g)(a){}^R\!{}^L\!I_{a,\omega}^{0}
\left(\dfrac{1}{\omega}\right)(t)\\
&=g(t)-\dfrac{(\omega g)(a)}{\omega(t)}.
\end{split}	
\end{equation*}
Now, we prove $(ii)$. According to Definition~\ref{PowerDIntegral}, one has
$$
{}^p\!I_{a,t,\omega}^{\alpha,\beta,p}\left({}^p{}^C 
\!D_{a,t,\omega}^{\alpha,\beta,p}g\right)(t)
=\phi(\alpha){}^p{}^C \!D_{a,t,\omega}^{\alpha,\beta,p}g(t)
+\ln p\cdot\psi(\alpha) {}^R\!{}^L\!I_{a,\omega}^{\beta}
\left({}^p{}^C \!D_{a,t,\omega}^{\alpha,\beta,p}g\right)(t).
$$
By applying Lemma~\ref{key1}, we obtain that
\begin{equation*}
\begin{split}
{}^p\!I_{a,t,\omega}^{\alpha,\beta,p}&\left({}^p{}^C 
\!D_{a,t,\omega}^{\alpha,\beta,p}g\right)(t)\\
&=\displaystyle\sum_{n=0}^{+\infty}\left(-\mu_{\alpha}
\ln p\right)^{n} {}^R\!{}^L\!I_{a,\omega}^{\beta n+1}
\left(\dfrac{(\omega g)'}{\omega}\right)(t)\\
&\quad+\mu_{\alpha}\ln p {}^R\!{}^L\!I_{a,\omega}^{\beta}
\left[\displaystyle\sum_{n=0}^{+\infty}\left(-\mu_{\alpha}
\ln p\right)^{n} {}^R\!{}^L\!I_{a,\omega}^{\beta n+1}
\left(\dfrac{(\omega g)'}{\omega}\right)(t)\right]\\
&=\displaystyle\sum_{n=0}^{+\infty}\left(-\mu_{\alpha}
\ln p\right)^{n} {}^R\!{}^L\!I_{a,\omega}^{\beta n+1}
\left(\dfrac{(\omega g)'}{\omega}\right)(t)
- \displaystyle\sum_{n=0}^{+\infty}\left(-\mu_{\alpha}
\ln p\right)^{n+1} {}^R\!{}^L\!I_{a,\omega}^{\beta (n+1)+1}
\left(\dfrac{(\omega g)'}{\omega}\right)(t)\\
&=\displaystyle\sum_{n=0}^{+\infty}\left(-\mu_{\alpha}
\ln p\right)^{n} {}^R\!{}^L\!I_{a,\omega}^{\beta n+1}
\left(\dfrac{(\omega g)'}{\omega}\right)(t)
- \displaystyle\sum_{n=1}^{+\infty}\left(-\mu_{\alpha}
\ln p\right)^{n} {}^R\!{}^L\!I_{a,\omega}^{\beta n+1}
\left(\dfrac{(\omega g)'}{\omega}\right)(t)\\
&={}^R\!{}^L\!I_{a,\omega}^{1}\left(\dfrac{(\omega g)'}{\omega}\right)(t)\\
&=\dfrac{1}{\omega(t)}\int_{a}^{t}(\omega g)'(x)\, \mathrm{d}x\\
&=g(t)-\dfrac{(\omega g)(a)}{\omega(t)}.
\end{split}	
\end{equation*}	
The proof is complete.
\end{proof}

\begin{remark}
Theorem~\ref{PFDandPFI} proves that the power fractional derivative 
and integral are commutative operators.	
\end{remark}

\begin{remark}
If we let $p=e$ in Theorem~\ref{PFDandPFI}, 
then we obtain the results presented in Theorem~3 of \cite{Hattaf1} 
for the generalized Hattaf fractional operators.
\end{remark}

As a corollary of our Theorem~\ref{PFDandPFI},  
we extend the Newton--Leibniz formula 
proved in \cite{Leibniz}.

\begin{corollary}
The power fractional derivative and integral satisfy the Newton--Leibniz formula
$$
{}^p{}^C \!D_{a,t,1}^{\alpha,\beta,p}
\left({}^p\!I_{a,t,1}^{\alpha,\beta,p}g\right)(t)
={}^p\!I_{a,t,1}^{\alpha,\beta,p}\left({}^p{}^C 
\!D_{a,t,1}^{\alpha,\beta,p}g\right)(t)
=g(t)-g(a).
$$
\end{corollary}

\begin{proof}
Follows from Theorem~\ref{PFDandPFI} 
with $\omega(t) \equiv 1$.
\end{proof}

% -----------------------------------------------------------------------

\section{Gronwall's inequality via PFI}
\label{Sec4}

In this section we establish a Gronwall's inequality in the framework 
of the power fractional integral. Our proof uses Lemma~\ref{GHattaf}.

\begin{theorem}
\label{theo1} 
Let $\alpha \in [0, 1)$, $\beta>0$, and $p>0$. 
Suppose $h$ and $u$ are non-negative 
and locally integrable functions on $[a,b)$, 
and $v$ is a non-negative, non-decreasing,
and continuous function on $[a,b)$ satisfying 
$v(t) \leq \lambda$, where $\lambda$ is a constant 
such that $1-\phi(\alpha)\lambda>0$. If
\begin{equation}
\label{CND1}
h(t)\leq u(t)+v(t){}^p\!I_{a,t,\omega}^{\alpha,\beta,p}h(t),
\end{equation}
then	
\begin{equation}
\label{Gronwall}
h(t)\leq \dfrac{u(t)}{1-\phi(\alpha)v(t)}
+ \int_{a}^{t}\sum_{n=1}^{+\infty}
\dfrac{( \ln p\cdot\psi(\alpha)v(t))^{n}u(s)(t-s)^{n\beta-1}}{
\Gamma(n\beta)(1-\phi(\alpha)v(t))^{n}(1-\phi(\alpha)v(s))}
\, \mathrm{d}s.
\end{equation}
\end{theorem}

\begin{proof}
By virtue of condition \eqref{CND1} 
and the PFI formula \eqref{PFI}, one has 
$$
h(t)\leq u(t)+\phi(\alpha)v(t)h(t)+\ln p
\cdot\psi(\alpha)v(t) {}^R\!{}^L\!I_{a,\omega}^{\beta}h(t),
$$
which leads to 
$$
h(t)\leq \dfrac{u(t)}{1-\phi(\alpha)v(t)}
+\dfrac{\ln p\cdot\psi(\alpha)v(t)}{1-\phi(\alpha)
v(t)} {}^R\!{}^L\!I_{a,\omega}^{\beta}h(t).
$$
Let $V(t)=\dfrac{\ln p\cdot\psi(\alpha)v(t)}{1-\phi(\alpha)v(t)}$. 
This function is non-negative and non-decreasing and, by applying 
the result of Lemma~\ref{GHattaf} with 
$U(t)=\dfrac{u(t)}{1-\phi(\alpha)v(t)}$, it follows that
$$
h(t)\leq U(t)+\int_{a}^{t}\sum_{n=1}^{+\infty} 
\dfrac{(V(t))^{n}}{\Gamma(n\beta)}
(t-s)^{n\beta-1}U(s)\, \mathrm{d}s.
$$   
Hence,
$$ 
h(t)\leq \dfrac{u(t)}{1-\phi(\alpha)v(t)}
+ \int_{a}^{t}\sum_{n=1}^{+\infty}
\dfrac{( \ln p\cdot\psi(\alpha)v(t))^{n}u(s)
(t-s)^{n\beta-1}}{\Gamma(n\beta)(1-\phi(\alpha)
v(t))^{n}(1-\phi(\alpha)v(s))}\, \mathrm{d}s,
$$
and the proof is complete.
\end{proof}

\begin{corollary}
\label{GronwallMettag}
Under the hypotheses of Theorem~\ref{theo1}, 
assume further that $v(t)$ is a non-decreasing 
function on $[a, b)$. Then,
$$
h(t)\leq \dfrac{u(t)}{1-\phi(\alpha)v(t)}
{}^p \! E_{\alpha,\beta}\left(
\dfrac{\psi(\alpha)v(t)  
(t-a)^{\beta}}{1-\phi(\alpha)v(t)}\right).
$$	
\end{corollary}

\begin{proof}
By virtue of inequality \eqref{Gronwall} and the assumption 
that $u(t)$ is a non-decreasing function on $[a, b)$, 
one may write that
\begin{equation*}
\begin{split}
h(t)&\leq \dfrac{u(t)}{1-\phi(\alpha)v(t)}
+ \dfrac{u(t)}{1-\phi(\alpha)v(t)} 
\int_{a}^{t}\sum_{n=1}^{+\infty}\dfrac{(\ln p\cdot\psi(\alpha)v(t))^{n}
(t-s)^{n\beta-1}}{\Gamma(n\beta)(1-\phi(\alpha)v(t))^{n}}\, \mathrm{d}s\\
&\leq \dfrac{u(t)}{1-\phi(\alpha)v(t)}\left(1
+\sum_{n=1}^{+\infty}\dfrac{( \ln p\cdot\psi(\alpha)v(t))^{n}}{
\Gamma(n\beta)(1-\phi(\alpha)v(t))^{n}}
\int_{a}^{t}(t-s)^{n\beta-1}\, \mathrm{d}s\right)\\
&\leq \dfrac{u(t)}{1-\phi(\alpha)v(t)}\left(1+
\sum_{n=1}^{+\infty}\dfrac{( \ln p\cdot\psi(\alpha)v(t))^{n}
(t-a)^{n\beta}}{\Gamma(n\beta)(1-\phi(\alpha)v(t))^{n}}\right).
\end{split}	
\end{equation*}
Therefore,	
$$
h(t)\leq \dfrac{u(t)}{1-\phi(\alpha)v(t)}
{}^p \! E_{\alpha,\beta}\left(\dfrac{\psi(\alpha)v(t)  
(t-a)^{\beta}}{1-\phi(\alpha)v(t)}\right),
$$
which completes the proof.	
\end{proof}

\begin{remark}
Our Gronwall's inequality for the power fractional integral, 
as given in Corollary~\ref{GronwallMettag}, includes, 
as particular cases, most of existing Gronwall's inequalities 
found in the literature that involve integrals 
with non-local and non-singular kernel, such us
\begin{itemize}
\item the Gronwall's inequality in the framework 
of the Atangana--Baleanu integral \cite{Jarad}, 
obtained when $p=e$, $\omega \equiv 1$ and $\beta=\alpha$;
\item the Gronwall's inequality in the framework of 
the generalized Hattaf fractional derivative \cite{KHattaf}, 
obtained when $p=e$.
\end{itemize}
\end{remark}

\begin{corollary}
\label{GronCor}
Let $\alpha \in [0, 1)$, $\beta > 0$, and $p>0$. 
Suppose that $h$ and $u$ are non-negative and locally 
integrable functions on $[a,b)$
and $v(t) \equiv \lambda$ be such that 
$1-\lambda\phi(\alpha)>0$. If
\begin{equation}
h(t)\leq u(t)+\lambda{}^p\!I_{a,t,\omega}^{\alpha,\beta,p}h(t),
\end{equation}
then	
$$
h(t)\leq \dfrac{u(t)}{1-\lambda\phi(\alpha)}
{}^p \! E_{\alpha,\beta}\left(\dfrac{\lambda\psi(\alpha)
(t-a)^{\beta}}{1-\lambda\phi(\alpha)}\right).
$$
\end{corollary}

% -----------------------------------------------------------------------

\section{Existence and uniqueness of solutions for power FDEs}
\label{Sec5}

In this section we study sufficient conditions for the
existence and uniqueness of solution to
the power fractional initial value problem
\begin{equation}
\label{system1}
{}^p{}^C \!D_{a,t,\omega}^{\alpha,\beta,p}y(t)
=f(t,y(t)), \quad
t\in [a,b]
\end{equation}
with
\begin{equation}
\label{CndInt}
y(a)=y_{0}, 
\end{equation}
where ${}^p{}^C \!D_{a,t,\omega}^{\alpha,\beta,p}$ denotes the PFD 
of order $\alpha$, defined by \eqref{PFD}, $f:[a,b]\times\mathbb{R}
\longrightarrow \mathbb{R}$ is a continuous nonlinear function with 
$f(a, y(a))=0$ and $y_{0}\in \mathbb{R}$ is the initial condition.

\begin{lemma}
\label{lem4}
A function $y\in C([a,b])$ is a solution of \eqref{system1}--\eqref{CndInt} 
if, and only if, it satisfies the integral equation	
\begin{equation}
\label{expsol}	
y(t)=\dfrac{\omega(a)}{\omega(t)}y_{0}+{}^p\!I_{a,t,\omega}^{\alpha,\beta,p}f(t,y(t)).
\end{equation}	
\end{lemma}

\begin{proof}
First, suppose that $y$ fulfills the integral formula \eqref{expsol}. Then,
$$
y(a)=y_{0}+{}^p\!I_{a,t,\omega}^{\alpha,\beta,p}f(a,y(a)).
$$
Since $f(a,y(a))=0$, we obtain that $y(a)=y_{0}$. 
Moreover, using the fact that $y(t)$ satisfies \eqref{expsol} 
and $(i)$ of Theorem~\ref{PFDandPFI}, it follows that
$$
{}^p{}^C \!D_{a,t,\omega}^{\alpha,\beta,p}y(t)
={}^p{}^C \!D_{a,t,\omega}^{\alpha,\beta,p}\left(
\dfrac{\omega(a)}{\omega(t)}y_{0}\right)
-\dfrac{\omega(a)f(a,y(a))}{\omega(t)}+f(t,y(t)),
$$
which implies that 
$$
{}^p{}^C \!D_{a,t,\omega}^{\alpha,\beta,p}y(t)= f(t,y(t)).
$$ 
Then $y(t)$ satisfies \eqref{system1}--\eqref{CndInt}.
	
Now, let us suppose that $y$ is a solution of the Cauchy problem 
\eqref{system1}--\eqref{CndInt}. Applying the power fractional 
integration operator to both sides of \eqref{system1}, 
and using formula $(ii)$ of Theorem~\ref{PFDandPFI}, we get
$$
y(t)=\dfrac{\omega(a)}{\omega(t)}y(a)
+{}^p\!I_{a,t,\omega}^{\alpha,\beta,p}f(t,y(t)).
$$
Therefore, since $y(a)=y_{0}$, we obtain 
formula \eqref{expsol}.
\end{proof}

\begin{theorem}
Let $y$ and $z$ be two solutions of system \eqref{system1}--\eqref{CndInt}. 
Assume that the function $f\in C([a,b]\times \mathbb{R},\mathbb{R})$ 
is Lipschitz in its second variable, that is, 
there exists a constant $L>0$ such that 
\begin{equation}
\label{Lipschitz}
|f(t,y)-f(t,z)|\leq L |y-z |, 
\quad \forall y,z \in \mathbb{R}
~ \text{ and } ~ t\in [a,b].
\end{equation}
If in addition $L<\dfrac{1}{\phi(\alpha)}$, then $y=z$. 
\end{theorem}

\begin{proof}
Let $y$ and $z$ be two solutions of problem \eqref{system1}--\eqref{CndInt}. 
By virtue of Lemma~\ref{lem4}, one has
$$
y(t)-z(t)={}^p\!I_{a,t,\omega}^{\alpha,\beta,p}\left(f(t,y(t)) -f(t,z(t))\right).
$$
Taking into account condition \eqref{Lipschitz}, it yields that
$$
|y(t)-z(t)|\leq L {}^p\!I_{a,t,\omega}^{\alpha,\beta,p}|y(t)-z(t)|.
$$	
By applying the result of Corollary~\ref{GronCor}, one obtains that
$$
|y(t)-z(t)|\leq \dfrac{0}{1-L\phi(\alpha)}
{}^p \! E_{\alpha,\beta}\left(\dfrac{L\psi(\alpha)
(t-a)^{\beta}}{1-L\phi(\alpha)}\right).
$$
It follows that $y=z$ for all $t \in [a,b]$.	
\end{proof}

\begin{theorem}
\label{theo3}
Assume that the function $f\in C([a,b]\times \mathbb{R},\mathbb{R})$ 
is Lipschitz in its second variable such that condition \eqref{Lipschitz} holds. If 	
\begin{equation}
\label{CndEU}
L \left(\phi(\alpha)+\dfrac{\ln p\cdot\psi(\alpha)
(b-a)^{\beta}}{\Gamma(\beta+1)}\right)<1,
\end{equation}		
then the Cauchy problem \eqref{system1}--\eqref{CndInt} 
has a unique solution.	
\end{theorem}

\begin{proof}
Let us define the operator 
$\Lambda: AC([ a, b])\longrightarrow AC([ a, b])$ as follows:
$$
(\Lambda y)(t)=\dfrac{\omega(a)}{\omega(t)}y(a)
+{}^p\!I_{a,t,\omega}^{\alpha,\beta,p}f(t,y(t)), 
\quad t\in [a, b].
$$
For all $y , z \in AC([ a, b])$ and $t \in [a, b ]$, one has 
\begin{equation*}
\begin{split}
|(\Lambda y)(t)-(\Lambda z)(t)|
&=|{}^p\!I_{a,t,\omega}^{\alpha,\beta,p}f(t,y(t))
-{}^p\!I_{a,t,\omega}^{\alpha,\beta,p}f(t,z(t))|\\
&\leq | \phi(\alpha)(f(t,y(t))-f(t,z(t)))
+ \ln p\cdot\psi(\alpha)\left({}^R\!{}^L\!I_{a,\omega}^{\beta}
f(t,y(t))-{}^R\!{}^L\!I_{a,\omega}^{\beta}f(t,z(t))\right)|\\
&\leq \phi(\alpha)|f(t,y(t))-f(t,z(t))|
+ \ln p\cdot\psi(\alpha){}^R\!{}^L\!I_{a,\omega}^{\beta }
\left|f(t,y(t))-f(t,z(t))\right|.
\end{split}	
\end{equation*}
Using the fact that $f$	satisfies the Lipschitz condition 
\eqref{Lipschitz}, we obtain that
\begin{equation*}
\begin{split}
|(\Lambda y)(t)-(\Lambda z)(t)|
&\leq L\phi(\alpha)|y-z|+ L \ln p
\cdot\psi(\alpha)|y-z|{}^R\!{}^L\!I_{a,\omega}^{\beta }(1)(t)\\
&\leq L\phi(\alpha)|y-z|+ L \ln p
\cdot\psi(\alpha)\dfrac{(t-a)^{\beta}}{\Gamma(\beta+1)}|y-z|.
\end{split}	
\end{equation*}
Therefore,
$$
\Arrowvert(\Lambda y)(t)-(\Lambda z)(t)\Arrowvert
\leq L\left( \phi(\alpha)+\ln p\cdot\psi(\alpha)
\dfrac{(b-a)^{\beta}}{\Gamma(\beta+1)}\right)
\Arrowvert y-z\Arrowvert.
$$
Hence, by virtue of \eqref{CndEU}, we conclude that $\Lambda$ 
is a contraction mapping. As a consequence of the Banach contraction 
principle, we conclude that system \eqref{system1} has a unique solution.	
\end{proof}

% ---------------------------------------------------

\section{Numerical analysis}
\label{NumAnaly}

Now we shall present a numerical method to approximate 
the solution of the nonlinear fractional differential equation
\eqref{system1} subject to \eqref{CndInt}, which is predicted
by Theorem~\ref{theo3}. Moreover, we also analyze 
the approximation error obtained from the new introduced 
scheme. Our main tool is the two-step Lagrange interpolation polynomial.

% --------

\subsection{Numerical scheme}

Consider the power nonlinear fractional differential equation 
\begin{equation}
\label{eqMain}
{}^p{}^C \!D_{a,t,\omega}^{\alpha,\beta,p}y(t) =f(t,y(t))
\end{equation}
subject to the given initial condition
$$
y(a)=y_{0}.
$$
From Theorem~\ref{PFDandPFI}, equation~\eqref{eqMain} 
can be converted into the fractional integral equation
$$
y(t)- \dfrac{\omega(a)}{\omega(t)}y(a)
=\phi(\alpha)f(t,y(t))+\ln p\cdot\psi(\alpha){}^p\!
I_{a,t,\omega}^{\alpha,\beta,p}f(t,y(t)),
$$
which implies that
\begin{equation}
\label{Err}
y(t)=\dfrac{\omega(a)}{\omega(t)}y(a)+\phi(\alpha)f(t,y(t))
+\dfrac{\ln p\cdot\psi(\alpha)}{\Gamma(\beta)}
\dfrac{1}{\omega(t)}\int_{a}^{t}(t-s)^{\beta-1}
\omega(s)f(s,y(s))\, \mathrm{d}s.
\end{equation}
Let $t_{n}=a+nh$ with $n\in\mathbb{N}$ 
and $h$ the discretization step. One has
\begin{equation*}
y(t_{n+1})=\dfrac{\omega(a)}{\omega(t_{n})}y(a)
+\phi(\alpha)f(t_{n},y(t_{n}))+\dfrac{\ln p
\cdot\psi(\alpha)}{\Gamma(\beta)}\dfrac{1}{\omega(t_{n})}
\int_{a}^{t_{n+1}}(t_{n+1}-s)^{\beta-1}\omega(s)
f(s,y(s))\, \mathrm{d}s,
\end{equation*}
which yields
\begin{equation}
\label{key2}
y(t_{n+1})=\dfrac{\omega(a)}{\omega(t_{n})}y(a)+\phi(\alpha)f(t_{n},y(t_{n}))
+\dfrac{\ln p\cdot\psi(\alpha)}{\Gamma(\beta)}\dfrac{1}{\omega(t_{n})}
\sum_{k=0}^{n}\int_{t_{k}}^{t_{k+1}}(t_{n+1}-s)^{\beta-1}g(s,y(s))
\, \mathrm{d}s
\end{equation}
with $g(s,y(s))=\omega(s)f(s,y(s))$. Function $g$ may 
be approximated over $[t_{k-1}, t_{k}]$, $k = 1, 2, \ldots, n$, 
by using the Lagrange interpolating polynomial that passes through 
the points $\left(t_{k-1}, g(t_{k-1}, y_{k-1})\right)$  
and $\left(t_{k}, g(t_{k}, y_{k})\right)$, as follows: 
\begin{equation}
\label{polyInter}
\begin{split}
P_{k}(s)&=\dfrac{s-t_{k}}{t_{k-1}-t_{k}}g(t_{k-1}, y(t_{k-1}))
+\dfrac{s-t_{k-1}}{t_{k}-t_{k-1}}g(t_{k}, y(t_{k}))\\
&\approx \dfrac{g(t_{k-1}, y_{k-1})}{h}(t_{k}-s)
+\dfrac{g(t_{k}, y_{k})}{h}(s-t_{k-1}).
\end{split}	
\end{equation}
Replacing the approximation \eqref{polyInter} 
in equation \eqref{key2}, we obtain that
\begin{equation}\label{mainEq}
\begin{split}
y_{n+1}
&=\dfrac{\omega(a)}{\omega(t_{n})}y_{0}
+\dfrac{\phi(\alpha)}{\omega(t_{n})}g(t_{n},y_{n})\\
&\qquad+\dfrac{\ln p\cdot\psi(\alpha)}{\Gamma(\beta)}
\dfrac{1}{\omega(t_{n})}\sum_{k=1}^{n}\left[ 
\dfrac{g(t_{k-1}, y_{k-1})}{h}\int_{t_{k}}^{t_{k+1}}
(t_{n+1}-s)^{\beta-1}(t_{k}-s)\, \mathrm{d}s\right.\\
&\qquad+\left.\dfrac{g(t_{k}, y_{k})}{h}
\int_{t_{k}}^{t_{k+1}}(t_{n+1}-s)^{\beta-1}
\left(s-t_{k-1}\right)\, \mathrm{d}s \right].
\end{split}	
\end{equation}
Moreover, we have
\begin{equation}
\label{Int1}
\int_{t_{k}}^{t_{k+1}}(t_{n+1}-s)^{\beta-1}(t_{k}-s)
\, \mathrm{d}s=\dfrac{h^{\beta+1}}{\beta(\beta+1)}
\left[(n-k)^{\beta}(n-k+1+\beta) -(n-k+1)^{\beta+1}\right]
\end{equation}
and
\begin{equation}
\label{Int2}
\int_{t_{k}}^{t_{k+1}}(t_{n+1}-s)^{\beta-1}(s-t_{k-1})\,\mathrm{d}s
= \dfrac{h^{\beta+1}}{\beta(\beta+1)}
\left[(n-k+1)^{\beta}(n-k+2+\beta)-(n-k)^{\beta}(n-k+2+2\beta)\right].
\end{equation}
The above equations \eqref{Int1} and \eqref{Int2} can then be included 
in equation \eqref{mainEq} to produce the following numerical scheme:
\begin{equation}
\label{mainEq2}
y_{n+1}=\dfrac{\omega(a)}{\omega(t_{n})}y_{0}
+\phi(\alpha)f(t_{n},y_{n})+\dfrac{\ln p
\cdot\psi(\alpha)h^{\beta}}{\Gamma(\beta+2)\omega(t_{n})}
\sum_{k=1}^{n} \omega(t_{k-1})f(t_{k-1}, y_{k-1})A^{\beta}_{n,k}
+\omega(t_{k})f(t_{k}, y_{k}) B^{\beta}_{n,k}
\end{equation}
with
$$ 
A^{\beta}_{n,k}=(n-k)^{\beta}(n-k+1+\beta) -(n-k+1)^{\beta+1}
$$
and
$$ 
B^{\beta}_{n,k}=(n-k+1)^{\beta}(n-k+2+\beta)-(n-k)^{\beta}(n-k+2+2\beta).
$$

\begin{remark}
The techniques used in this section are similar to the ones in \cite{HatNum} 
for the generalized Hattaf fractional derivative and in \cite{ToufAtan} 
for the Atangana--Baleanu fractional derivative.
\end{remark}

% --------

\subsection{Error analysis}

We now examine the numerical error of our developed approximation scheme~\eqref{mainEq2}.

\begin{theorem} 
Let \eqref{eqMain} be a nonlinear power fractional differential equation, 
such that $g=\omega f$ has a bounded second derivative. Then, 
the approximation error is estimated to verify
$$
\mid R^{\alpha,\beta,p}_{n}\mid
\leq\dfrac{\ln p\cdot\psi(\alpha)h^{\beta+2}}{4\Gamma(\beta+2)
\omega(t_{n})}(n+1)(n+4+2\beta)\left[(n+1)^{\beta}-\beta n^{\beta}\right]
\underset{s\in [a,t_{n+1}]}{\max}\mid g^{(2)}(s,y(s))\mid.
$$
\end{theorem}

\begin{proof} 
From \eqref{Err}, one has
\begin{equation}
\label{key3}
y(t_{n+1})=\dfrac{\omega(a)}{\omega(t_{n})}y(a)
+\phi(\alpha)f(t_{n},y(t_{n}))+\dfrac{\ln p
\cdot\psi(\alpha)}{\Gamma(\beta)}\dfrac{1}{\omega(t_{n})}
\sum_{k=0}^{n}\int_{t_{k}}^{t_{k+1}}
(t_{n+1}-s)^{\beta-1}g(s,y(s))\, \mathrm{d}s.
\end{equation}
Therefore,
\begin{multline*}
y(t_{n+1})
=\dfrac{\omega(a)}{\omega(t_{n})}y(a)+\phi(\alpha)f(t_{n},y(t_{n}))\\
+\dfrac{\ln p\cdot\psi(\alpha)}{\Gamma(\beta)}\dfrac{1}{\omega(t_{n})}
\sum_{k=0}^{n}\int_{t_{k}}^{t_{k+1}}(t_{n+1}-s)^{\beta-1}\left[P_{k}(s)
+\dfrac{(s-t_{k})(s-t_{k-1})}{2!}[g^{(2)}(s,y(s))]_{s=\xi_{s}}\right]
\, \mathrm{d}s,
\end{multline*}
which implies that
\begin{equation*}
\begin{split}
y(t_{n+1})
&=\dfrac{\omega(a)}{\omega(t_{n})}y(a)+\phi(\alpha)f(t_{n},y(t_{n}))\\
&\quad+\phi(\alpha)f(t_{n},y_{n})+\dfrac{\ln p\cdot\psi(\alpha)
h^{\beta}}{\Gamma(\beta+2)\omega(t_{n})}
\sum_{k=0}^{n} g(t_{k-1}, y_{k-1})A^{\beta}_{n,k}+g(t_{k}, y_{k})
B^{\beta}_{n,k}+R^{\alpha,\beta,p}_{n}
\end{split}	
\end{equation*}
with the remainder
$$
R^{\alpha,\beta,p}_{n}
=\dfrac{\ln p\cdot\psi(\alpha)}{\Gamma(\beta)}\dfrac{1}{\omega(t_{n})}
\sum_{k=0}^{n}\int_{t_{k}}^{t_{k+1}}(t_{n+1}-s)^{\beta-1}
\dfrac{(s-t_{k})(s-t_{k-1})}{2!}[g^{(2)}(s,y(s))]_{s=\xi_{s}}
\, \mathrm{d}s.
$$
Using the fact that function $s \mapsto (s-t_{k-1})(t_{n+1}-s)$ 
is positive on the interval $[t_{k}, t_{k+1}]$, it follows that 
there exists a $\xi_{k} \in [t_{k}, t_{k+1}]$ such that
$$
R^{\alpha,\beta,p}_{n}
=\dfrac{\ln p\cdot\psi(\alpha)}{\Gamma(\beta)}
\dfrac{1}{\omega(t_{n})}\sum_{k=0}^{n}g^{(2)}(\xi_{k},y(\xi_{k}))
\dfrac{(\xi_{k}-t_{k})}{2}
\int_{t_{k}}^{t_{k+1}}(t_{n+1}-s)^{\beta-1}(s-t_{k-1}) \, \mathrm{d}s.
$$
Using \eqref{Int2}, we obtain that
$$
R^{\alpha,\beta,p}_{n}=\dfrac{\ln p
\cdot\psi(\alpha)h^{\beta+1}}{2\Gamma(\beta+2)\omega(t_{n})}
\sum_{k=0}^{n}g^{(2)}(\xi_{k},y(\xi_{k}))(\xi_{k}-t_{k})B^{\beta}_{n,k}.
$$
Therefore,
$$
\mid R^{\alpha,\beta,p}_{n}\mid\leq\dfrac{\ln p\cdot\psi(\alpha)
h^{\beta+2}}{2\Gamma(\beta+2)\omega(t_{n})}
\underset{s\in [a,t_{n+1}]}{\max}\mid g^{(2)}(s,y(s))
\mid\cdot\left|\sum_{k=0}^{n}B^{\beta}_{n,k}\right|.
$$
Then, from formulas
\begin{equation*}
\begin{split}
B^{\beta}_{n,k}
&=(n-k+2+\beta)\left[(n-k+1)^{\beta}-\beta(n-k)^{\beta}\right]\\
&\leq(n-k+2+\beta)\left[(n+1)^{\beta}-\beta n^{\beta}\right]
\end{split}	
\end{equation*}
and
$$
\sum_{k=0}^{n}(n-k+2+\beta)=\dfrac{(n+1)(n+4+2\beta)}{2},
$$
we deduce that
$$
\mid R^{\alpha,\beta,p}_{n}\mid
\leq\dfrac{\ln p\cdot\psi(\alpha)h^{\beta+2}}{4\Gamma(\beta+2)
\omega(t_{n})}(n+1)(n+4+2\beta)\left[(n+1)^{\beta}
-\beta n^{\beta}\right]\underset{s\in [a,t_{n+1}]}{\max}
\mid g^{(2)}(s,y(s))\mid,
$$
which completes the proof.
\end{proof}

% -----------------------------------------------------------------------

\section{Examples and simulation results}
\label{Sec7}

In this section, we begin by illustrating the suggested numerical method 
of Section~\ref{NumAnaly} with a power FDE for which we can compute 
its exact solution. Then, as a second example, we apply our main analytical 
and numerical results to a nonlinear power FDE 
for which no exact solution is known.

\begin{example}
Let us consider the following power fractional equation:	
\begin{equation}
\label{systema}
{}^p{}^C \!D_{0,t,\omega}^{\alpha,\beta,p}y(t)
= t^{2},\quad t\in [0,10]
\end{equation}
subject to
\begin{equation}
\label{ConExp1}
y(0)=0,
\end{equation} 	
where $\omega(t) \equiv 1$. By applying the power fractional integral 
to both sides of \eqref{systema} and using formula $(ii)$ 
of Theorem~\ref{theo1}, we obtain the exact solution 
of~\eqref{systema}--\eqref{ConExp1}, which is given by
\begin{equation}
y(t)=\phi(\alpha)t^{2}
+ \frac{2\ln p\cdot\psi(\alpha)}{\Gamma(\beta+3)}t^{\beta+2}.
\end{equation}
We now apply the developed numerical scheme \eqref{mainEq2} 
to approximate the solution of \eqref{systema}--\eqref{ConExp1}. 
For numerical simulations, we choose the normalization function 
$$
N(\alpha)=1-\alpha+\dfrac{\alpha}{\Gamma(\alpha)}.
$$
The comparison between the exact and approximate solutions of 
\eqref{systema}--\eqref{ConExp1} is depicted in Figures~\ref{Figure1} 
and \ref{Figure2}. 
% ------------------------------------------------------------------
\begin{figure}[H]
\begin{center}	
\begin{minipage}{0.46\linewidth}
\centering
\includegraphics[scale=0.57]{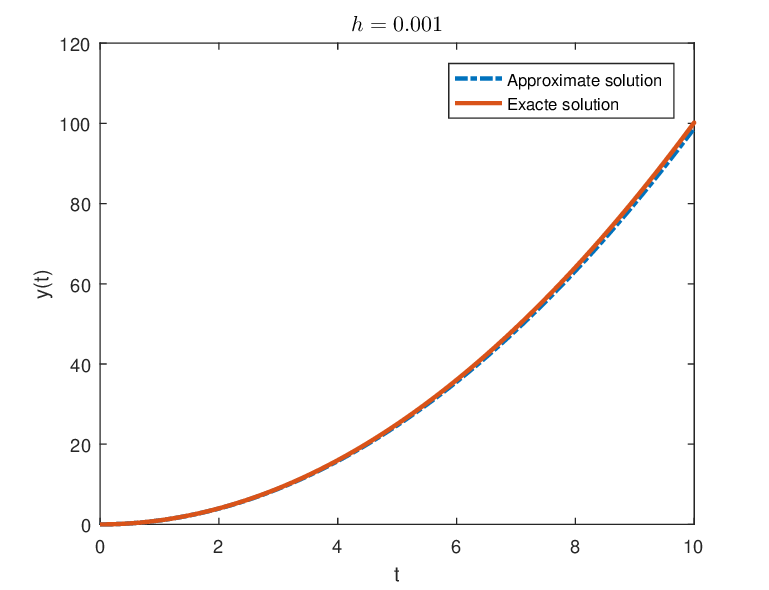}
\end{minipage}
\begin{minipage}{0.46\linewidth}
\centering
\includegraphics[scale=0.57]{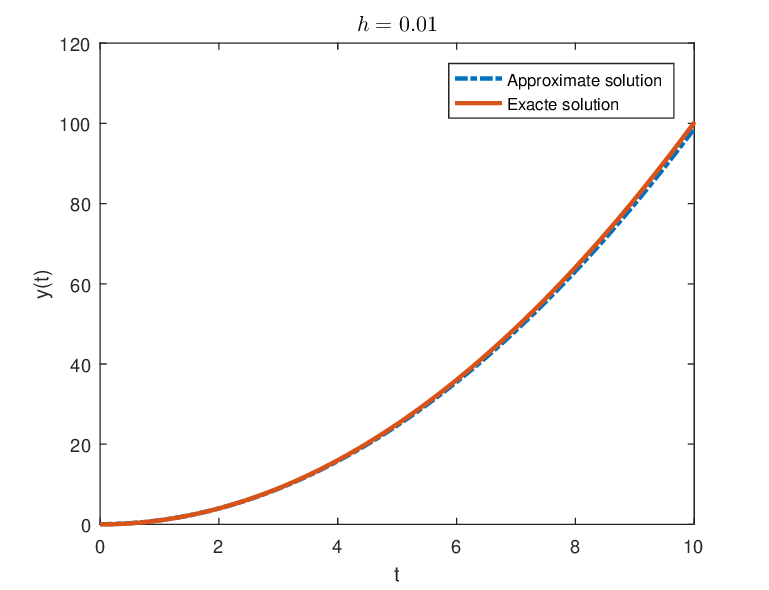}
\end{minipage}
\begin{minipage}{0.46\linewidth}
\centering
\includegraphics[scale=0.57]{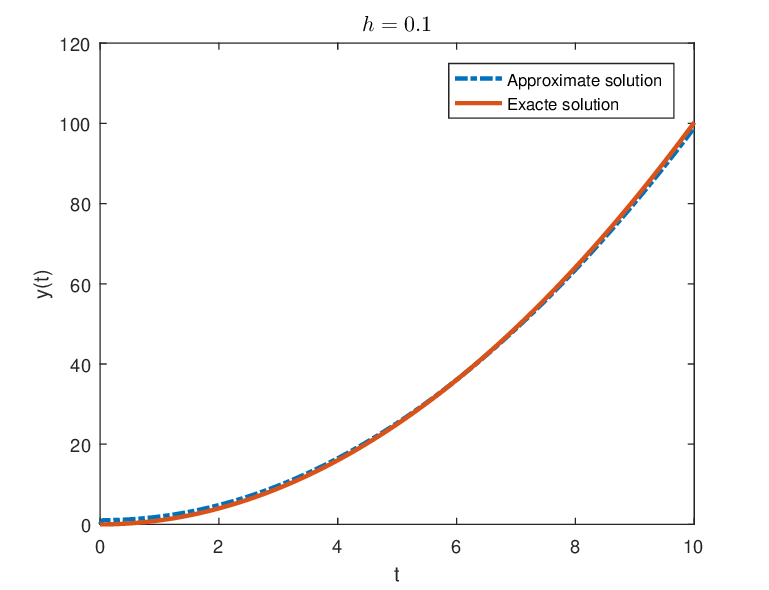}
\end{minipage}
\caption{Approximate and exact solutions of system 
\eqref{systema}--\eqref{ConExp1} for $\alpha=0.1$, 
$\beta=0.2$, $p=1.1$ and different values of the 
discretization step $h$.}
\label{Figure1}
\end{center}
\end{figure}
% ------------------------------------------------------------------
\begin{figure}[H]
\begin{center}	
\begin{minipage}{0.46\linewidth}
\centering
\includegraphics[scale=0.57]{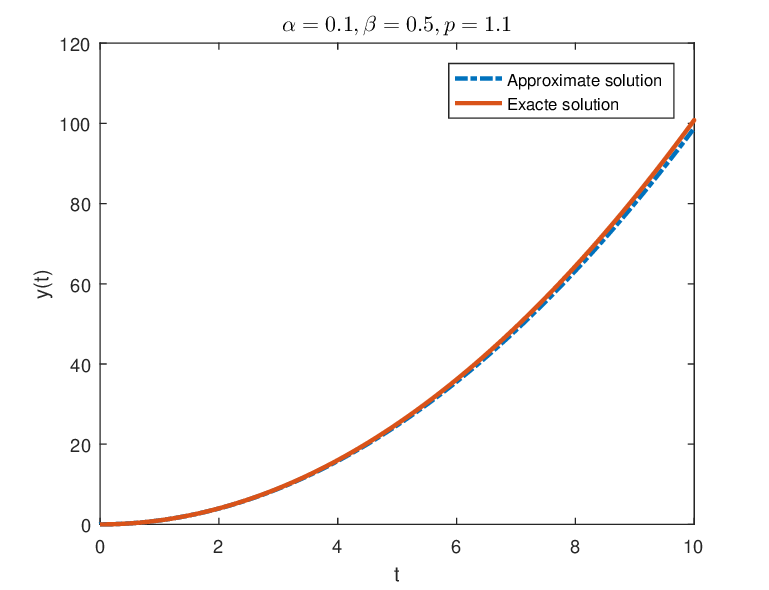}
\end{minipage}
\begin{minipage}{0.46\linewidth}
\centering
\includegraphics[scale=0.57]{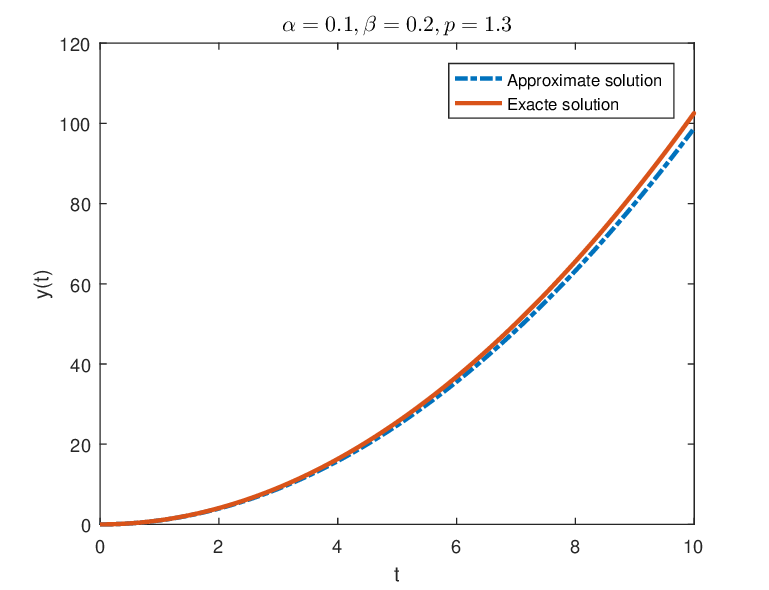}
\end{minipage}
\begin{minipage}{0.46\linewidth}
\centering
\includegraphics[scale=0.57]{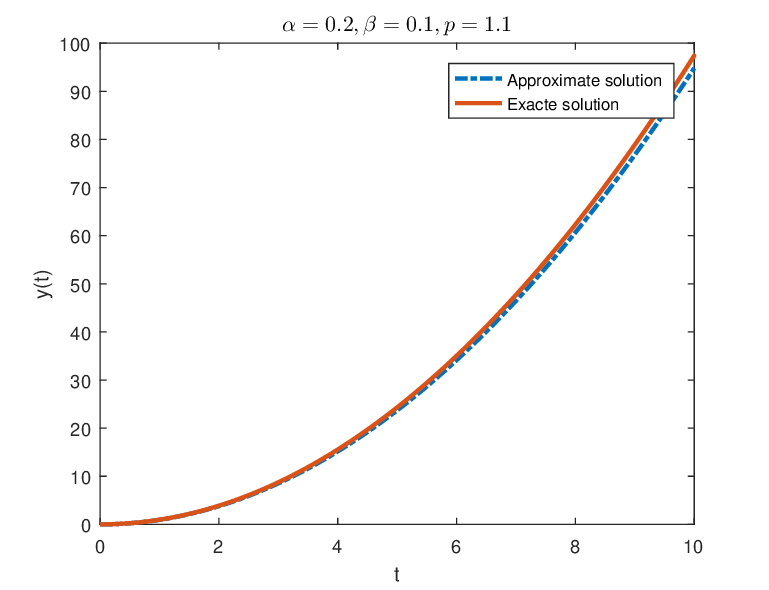}
\end{minipage}
\caption{Approximate and exact solutions of system \eqref{systema}--\eqref{ConExp1} 
for different values of $\alpha$, $\beta$ and $p$, with $h=0.001$.} 
\label{Figure2}
\end{center}
\end{figure}
% ------------------------------------------------------------------

The maximum error of the numerical approximations 
is given in Table~\ref{Table1}, for $\alpha=0.1$, 
$\beta=0.2$, $p=1.1$ and different values 
of the discretization step $h$.
% ------------------------------
\begin{table}[H]
\center
\caption{Maximum error corresponding to different values of $h$ 
with $\alpha=0.1$, $\beta=0.2$ and $p=1.1$.}
\label{Table1}
\begin{tabular}{l c c} \\ \hline
Discretization step ($h$) &\qquad Approximation error\\ \hline	
$h=0.1$ &\qquad $2.447\times 10^{-1}$\\
$h=0.01$  &\qquad  $ 2.080\times 10^{-2}$\\
$h=0.001$  &\qquad  $ 6.100 \times 10^{-3}$\\ \hline
\end{tabular}
\end{table}	
% ------------------------------

From Figures~\ref{Figure1} and \ref{Figure2}, we observe that 
the proposed numerical method gives a good agreement between 
the exact and approximate solutions for different value of 
$\alpha$, $\beta$, $p$ and the discretization step $h$. 
Table~\ref{Table1} shows that the convergence 
of the numerical approximation depends on the step of discretization~$h$. 
By comparing the exact and approximate solutions, we conclude 
that the new proposed numerical scheme is very efficient 
and converges quickly to the exact solution.
\end{example}

\begin{example}
Consider the following nonlinear power fractional differential equation:
\begin{equation}
\label{system2}
{}^p{}^C \!D_{0,t,\omega}^{\alpha,1,e}y(t)
=\dfrac{t^{2}}{15}\left( \dfrac{\cos(2t)}{1+|y(t)|}\right), 
\quad t\in [0,4]
\end{equation}
subject to
\begin{equation}
\label{CondEx2}   
y(0)=\sqrt\pi.
\end{equation}
This example is a particular case of problem \eqref{system1}--\eqref{CndInt} 
with $\beta=1$, $p=e$, $y_{0}=\sqrt\pi$, $a=0$, $b=4$ and 
$$
f(t,y(t))
=\dfrac{t^{2}}{15}\left(\dfrac{\cos(2t)}{1+|y(t)|}\right)
$$ 
with $f(0,y(0))=0$. Here, we choose the normalization function $N(\alpha)=1$.
	
For all  $y, z \in\mathbb{R}$ and $t\in[0, 4]$, one has 
\begin{equation*}
\begin{split}
|f(t,y)-f(t,z)|
&=\dfrac{t^{2}}{15}|\cos(3t^{2})|\left(\mid\dfrac{1}{1+|y|}
-\dfrac{1}{1+|z|}\mid\right)\\
&\leq \dfrac{1}{15}\left(|z|-|y|\right)\\
&\leq \dfrac{1}{15}|y-z|.
\end{split}	
\end{equation*}
Thus, function $f$ is continuous and satisfies the Lipschitz condition 
\eqref{Lipschitz} with $L=\dfrac{1}{15}$. Moreover, for any $\alpha \in [0,1)$, 
we have $\phi(\alpha)=1-\alpha$, $\psi(\alpha)=\alpha$ and 
\begin{equation*}
L \left(\phi(\alpha)+\dfrac{\ln p\cdot\psi(\alpha)
(b-a)^{\beta}}{\Gamma(\beta+1)}\right)
=\dfrac{1}{15} \left(1+3\alpha\right)<1.
\end{equation*}	
Hence, condition \eqref{CndEU} holds. Then, by applying Theorem~\ref{theo3}, 
it follows that problem \eqref{system2}--\eqref{CondEx2} 
has a unique solution on $[0, 4]$.
	
We now use our proposed method to solve the system \eqref{system2}--\eqref{CondEx2}. 
For numerical simulations, we take the weight function $\omega(t)=t+2$.
	
The approximate solution of \eqref{system2}--\eqref{CondEx2} is displayed 
in Figures~\ref{figa} and \ref{figb} for different values of 
$\alpha$, $\beta=1$ and $p=e$, using two discretization steps: 
$h=0.1$ and $h=0.01$.
% ---------------------------------------
\begin{figure}[H]
\centering
\includegraphics[scale=0.65]{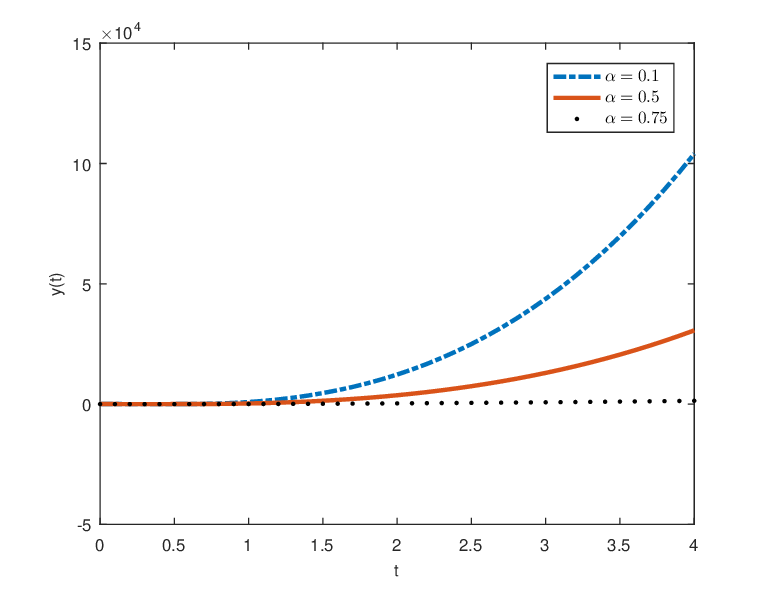}
\caption{Approximate solution of system \eqref{system2}--\eqref{CondEx2} 
for  different values of $\alpha$, $\beta=1$, $p=e$ and $h=0.1$.}
\label{figa}
\end{figure}
% ---------------------------------------
\begin{figure}[H]
\centering
\includegraphics[scale=0.65]{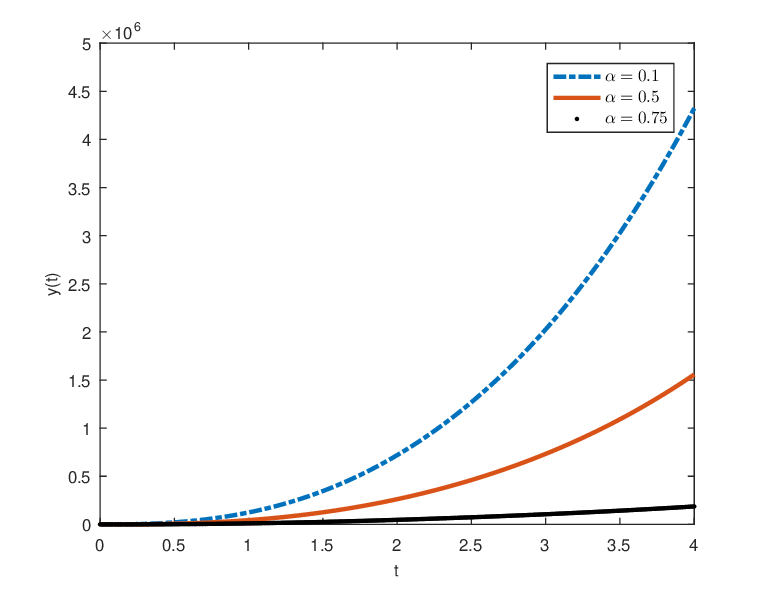}
\caption{Approximate solution of system \eqref{system2}--\eqref{CondEx2} 
for different values of $\alpha$, $\beta=1$, $p=e$ and $h=0.01$.}
\label{figb}
\end{figure}
% ---------------------------------------
\end{example}

% ---------------------------------------------------

\section{Conclusion}
\label{sec:conc}

In this paper, (i) we established a new formula 
for the power fractional derivative with a non-local and non-singular 
kernel in the form of an infinite series of the standard weighted Riemann--Liouville 
fractional integral. This brings out more clearly the non-locality properties 
of the fractional derivative and makes it easier to handle certain computational aspects. 
By means of the proposed formula, we derived useful properties of the power fractional operators, 
for example the Newton--Leibniz formula has been rigorously extended. 
(ii) We presented a new version of Gronwall's inequality via the power fractional 
integral, which includes many versions of Gronwall's inequality 
found in the literature, such us the generalized Hattaf and Atangana--Baleanu 
fractional Gronwall's inequalities. (iii) We proved the existence and uniqueness 
of solutions to nonlinear power fractional differential equations 
using the fixed point principle; and, based on Lagrange 
polynomial interpolation, (iv) we provided a new explicit 
numerical method to approximate the solutions of power FDEs 
with the approximation error being also examined. However, 
we only presented a bound for the error and the proof
of the convergence of the numerical scheme is still an open problem.
Numerical examples and simulation results were discussed and show 
that our developed method is very efficient, highly accurate, 
and converges quickly.

As future work, we aim to apply our obtained analytical 
and numerical results to develop power fractional models describing 
real world phenomena such us the world population growth 
and the dynamics of an epidemic disease. This issue is currently 
under investigation and will appear elsewhere.

% ---------------------------------------------------

\section*{Acknowledgements}

Zitane and Torres are supported by The Center 
for Research and Development in Mathematics and Applications (CIDMA)
through the Portuguese Foundation for Science and Technology 
(FCT -- Funda\c{c}\~{a}o para a Ci\^{e}ncia e a Tecnologia),
project UIDB/04106/2020. Zitane is also grateful to the post-doc 
fellowship at CIDMA-DMat-UA, reference UIDP/04106/2020.

% ---------------------------------------------------

% ---------------------------------------------------

\end{document}